\documentclass[a4paper,11pt]{article}
\usepackage{geometry,latexsym,amssymb,amsmath,amsthm,color,bm}
\usepackage[latin5]{inputenc}
\usepackage{enumerate}
\usepackage{enumitem}
\usepackage{authblk}
\usepackage[all]{xy}
\usepackage[bookmarks=true,
bookmarksnumbered=true, breaklinks=true,
pdfstartview=FitH, hyperfigures=false,
plainpages=false, naturalnames=true,
colorlinks=true,
pdfpagelabels]{hyperref}
\usepackage{palatino}

\newtheorem{example}{Example}[section]
\newtheorem{Def}[example]{Definition}
\newtheorem{Exam}[example]{Example}
\newtheorem{Prop}[example]{Proposition}
\newtheorem{Theo}[example]{Theorem}
\newtheorem{Lem}[example]{Lemma}
\newtheorem{Rem}[example]{Remark}
\newtheorem{Cond}[example]{Condition}

\def\inc{\widetilde{\in}}

\def\cIm{\operatorname{cIm}}
\def\Aut{\operatorname{Aut}}

\def\cGr{\mathsf{cGr}}
\def\Sets{\mathsf{Sets}}

\def\qed{\hfill $\Box$}
\def\Ker{\operatorname{Ker}}
\def\cKer{\operatorname{cKer}}
\def\NSCM/(A,B,\mu){\mathsf{NSCM/(A,B,\mu)}}
\def\NSGGd/G{\mathsf{NSGGd/G}}

\def\epsilon{\varepsilon}

\def\q{\quad}
\def\t{\text}

\begin{document}
\title{\bf Groups up to congruence relation and from categorical groups to c-crossed modules}

\author[1]{Tamar Datuashvili\thanks{\textbf{Corresponding author:} Tamar Datuashvili (e-mail: tamar.datu@gmail.com)}}
\affil[1]{\small{A. Razmadze Mathematical Institute of I. Javakhishvili Tbilisi State University, 6 Tamarashvii Str., Tbilisi 0177, Georgia}}
\author[2]{Osman Mucuk\thanks{O. Mucuk (e-mail : mucuk@erciyes.edu.tr)}}
\affil[2]{\small{Department of Mathematics, Erciyes University, 38039, Kayseri, TURKEY}}
\author[3]{Tun\c{c}ar \c{S}AHAN\thanks{T. \c{S}ahan (e-mail : tuncarsahan@gmail.com)}}
\affil[3]{\small{Department of Mathematics, Aksaray University, 68100, Aksaray, TURKEY}}

\date{}
\maketitle 

\begin{abstract}
We introduce a notion of c-group, which is a group up to congruence relation and consider the corresponding category. Extensions, actions and crossed modules (c-crossed modules) are defined in this category and the semi-direct product is constructed. We prove that each categorical group gives rise to c-groups and to a c-crossed module, which is a connected, special and strict c-crossed module in the sense defined by us. The results obtained here will be applied in the proof of an equivalence of the categories of categorical groups and connected, special and strict c-crossed modules.\\[3mm]
\noindent \footnotesize{\textbf{Mathematics Subject Classifcation (2010):} 20L99 $\cdot$ 20L05 $\cdot$ 18D35.}
\end{abstract}

\section{Introduction}

Our aim was to obtain for categorical groups an analogous description in terms of certain crossed module type objects as we have it for $\mathcal{G}$-groupoids  obtained by Brown and Spencer \cite{BS1}, which are strict categorical groups, or equivalently, group-groupoids or internal categories in the category of groups. Under categorical groups we mean a coherent 2-group in the sense of Baez and Lauda \cite{baez-lauda-2-groups}. It is important to note that it is well known that a categorical group is equivalent to a strict categorical group \cite{Sinh 2,JoyalStreet,baez-lauda-2-groups}, but we do not have an equivalence  between the corresponding categories.  This idea brought us to a new  notion of group up to congruence relation.  In this way we came to the definition of c-group and the corresponding category. Then we defined action in this category and introduced a notion of c-crossed module. Among of this kind of objects we distinguished connected, strict and special c-crossed modules denoted as cssc-crossed module. We proved that every categorical group gives rise to a cssc-crossed module. In the following paper we will prove that there is an equivalence between the category of categorical groups and the category of cssc-crossed modules. We hope that this result will give a chance  to consider for categorical groups the problems analogous to those considered and solved in the case of strict categorical groups in terms of group-groupoids and internal categories in \cite{BM1,Dat1,Dat2,Dat3,Dat4}.

In Section 2 we give a definition of a categorical group, group groupoid and crossed module in the category of groups. In Section 3  we give a definition of  group up to congruence relation, shortly c-group, give examples and consider the corresponding category of c-groups denoted as $\cGr$. We define $\cKer f$, $\cIm f,$ for any morphism $f$ in $\cGr,$ and normal c-subgroup in any c-group. In Section 4 we define split extension and action in $\cGr$. After this we give a definition of c-crossed module and give examples. We introduce the notions of special, strict and connected c-crossed modules and give examples. We prove that every categorical group defines a cssc-crossed module.

\section{Preliminaries}
Recall the definition of a monoidal category given by Mac Lane \cite{Mac71}.
\begin{Def}\label{moncat}
	A \textbf{monoidal category} is a category $\mathsf{C}=(C_0, C_1, d_0, d_1, i, m)$ equipped with a bifunctor $+\colon\mathsf{C}\times\mathsf{C}\rightarrow\mathsf{C}$ called the monoidal sum, an object $0$ called the zero object, and three natural isomorphisms $\alpha, \lambda$ and $\rho$. Explicitly, \[\alpha=\alpha_{x,y,z}\colon (x+y)+z\stackrel{\approx}{\rightarrow} x+(y+z)\] is natural for all $x,y,z\in C_0$, and the pentagonal diagram
	\[\xymatrix{
		((x+y)+z)+t  \ar[d]_{\alpha+1} \ar[r]^-{\alpha} & (x+y)+(z+t)  \ar[r]^-{\alpha} & x+(y+(z+t)) \\
		(x+(y+z))+t  \ar[rr]_-\alpha & & x+((y+z)+t) \ar[u]^{1+\alpha} }\]
	commutes for all $x,y,z,t\in C_0$. Again, $\lambda$ and $\rho$ are natural $\lambda_x\colon 0+x\stackrel{\approx}{\rightarrow} x,$ $\rho_x\colon x+0\stackrel{\approx}{\rightarrow} x,$ for all $x\in C_0$, the diagram
	\[\xymatrix{
		(x+0)+y  \ar[r]^-{\alpha} \ar[d]_-{\rho+1} &  x+(0+y)  \ar[d]^-{1+\lambda} \\
		x+y  \ar@{=}[r] & x+y }\]
	commutes for all $x,y\in C_0$ and also $\lambda_0=\rho_0\colon 0+0 \stackrel{\approx}{\rightarrow} 0$. Moreover "all" diagrams involving $\alpha,\lambda$, and $\rho$ must commute.
\end{Def}

In this definition we use the term monoidal sum and denote it as $+,$ instead of monoidal product, used in the original definition, and write the operation additively. From the definition it follows that $1_0+f\approx f+1_0 \approx f,$ for any morphism $f.$
In what follows the isomorphisms $\alpha, \lambda$ and $\rho$ involved in group like identities, their inverses, compositions and their monoidal sums will be called as \emph{special isomorphisms}.
Since $+$ is a bifunctor in a monoidal category we have $d_j(f+g)=d_j(f)+d_j(g), j=0,1, i(x+y)=i(x)+i(y)$ and interchange law $(f'+g')(f+g)=f'f+g'g,$ whenever the composites $f'f$ and $g'g$ are defined, for any $x, y\in C_0, f,f',g,g'\in C_1$.

Any category $\mathsf{C}$ with finite products can be considered as a monoidal category where $+$ assigns any two objects to their product and $0$ is the terminal object. The category of abelian groups $\mathsf{Ab}$ is a monoidal category where tensor product $\otimes$ is monoidal sum and $\mathbb{Z}$ is the unit object. There are other examples as well \cite{Mac71}.

In a monoidal category, if the special isomorphisms $\alpha, \lambda$, and $\rho$ are identities, then $\mathsf{C}$ is called a \textbf{strict monoidal category}.

Let $\mathsf{C}$ and $\mathsf{C'}$ be two monoidal categories. A \emph{(strict) morphism} of monoidal categories $T:(C,+,0,\alpha,\lambda,\rho)\rightarrow (C',+',0',\alpha',\lambda',\rho')$ is a functor $T:C\rightarrow C'$, such that for all objects $x, y, z$ and morphisms $f$ and $g$ there are equalities $T(x+y)=Tx+'Ty, T(f+g)=Tf+'Tg, T0=0', T\alpha_{x,y,z}=\alpha'_{Tx,Ty,Tz}, T\lambda_x=\lambda'_{Tx}, T\varrho_x=\rho'_{Tx}.$

\begin{Def}\label{invobj}\cite{baez-lauda-2-groups}
	If $x$ is an object in a monoidal category, an inverse for $x$ is an object $y$ such that $x+y\approx 0$ and $y+x\approx 0.$ If $x$ has an inverse, it is called invertible.
\end{Def}

As it is noted in \cite{baez-lauda-2-groups,JoyalStreet} and it is easy to show, that if any object has a one side inverse in a monoidal category, then any object is invertible.

\begin{Def}\label{catgrp}
	A \textbf{categorical group} $\mathsf{C}=(C_0, C_1, d_0, d_1, i, m)$ is a monoidal groupoid,  where all objects are invertible and moreover, for every object $x\in C_0$ there is an object $-x\in C_0$ with a family of natural isomorphisms
	\[\epsilon_x\colon -x+x\approx 0,\]
	\[\delta_x\colon x+(-x)\approx 0\] such that the following diagrams are commutative:
	\[\xymatrix@C=2.3pc{
		0 + x \ar[r]^<<<<<<{\delta^{-1}_x + 1} \ar[d]_{\lambda _x}
		& (x + (-x)) + x \ar[r]^<<<<<{a_{x, -x ,x}}
		& x+( -x + x) \ar[d]^{1_x + \epsilon_x} \\
		x \ar[rr]_{\rho ^{-1}_x}
		& & x + 0  }
	\]
	
	\[
	\xymatrix@C2.3pc{
		-x + 0
		\ar[r]^<<<<<<{1 + \delta^{-1}_x}
		\ar[d]_{\rho _{-x}}
		& -x + (x + (-x))
		\ar[r]^{a^{-1}_{-x, x, -x}}
		& (-x + x)+ (-x)
		\ar[d]^{\epsilon_x + 1_{-x}} \\
		-x
		\ar[rr]_{\lambda^{-1}_{-x}}
		&& 0 + (-x)   }
	\]
\end{Def}

It is important, and  a well-known fact, that from the definition of categorical group it follows that for any morphism $f:x\rightarrow x' \in C_1$ there is a morphism $-f\colon -x\rightarrow -x'$ with natural isomorphisms $-f+f\approx 0$ and $f+ (-f)\approx 0,$ where $0$ morphism is $1_0$ (see e.g. \cite{Sinh 1}).
Like the case of monoidal category the natural transformations $\alpha,\lambda$, $\rho, \epsilon, \delta,$ identity transformation $1_ \mathsf{C}\rightarrow 1_
\mathsf{C},$ their compositions and sums will be called \emph{special isomorphisms}. Categorical group defined above is coherent \cite{Laplaza,baez-lauda-2-groups}, which means that all diagrams commute involving special isomorphisms. For a monoidal category one can see in \cite{Mac71}, Coherence Chapter VII Section 2.

A categorical group is called strict if the special isomorphisms $\alpha$, $\lambda$, $\rho$, $\epsilon,$ and $\delta$ are identities. Strict categorical groups are known as group-groupoids (see below the definition), internal categories in the category of groups or 2-groups in the literature.

The definition of categorical group we gave is Definition 7 by Baez and Lauda in \cite{baez-lauda-2-groups}, where the operation is multiplication and which is called there coherent 2-group. Sinh \cite{Sinh 1} calls them ``gr-categories"  and this name is also used by other authors as well, e.g. Breen \cite{Breen}. It is called ``categories with group structure" by Ulbrich \cite{Ulbrich} and Laplaza \cite{Laplaza} in which all morphisms are invertible. The term categorical group for strict categorical groups is used by Joyal and Street \cite{JoyalStreet}, and it is used by Vitale \cite{Vitale1,Vitale2} and others for non strict ones.

From the functorial properties of addition $+$ it follows that in a categorical group we have $-1_x=1_{-x},$ for any $x\in C_0.$ Since an isomorphism between morphisms $\theta :f\thickapprox g$ means that there exist isomorphisms $\theta_i:d_i(f)\rightarrow d_i(g), i=0,1$ with $\theta_1 f=g\theta_0$, from natural property of special isomorphisms there exist special isomorphisms between the morphisms in $C_1$. But if $\theta_i, i=0,1$ are special isomorphisms, it doesn't imply that $\theta$ is a special isomorphism; in this case we will call $\theta$ \emph{weak special isomorphism}. It is obvious that a special isomorphism between the morphisms in $C_1$ implies the weak special isomorphism. Note that if $f\approx f'$ is a weak special isomorphism, then from the coherence property it follows that $f'$ is a unique morphism weak special isomorphic to $f$ with the same domain and codomain objects as $f'$.

\begin{Exam} Let $X$ be a topological space and $x\in X$ be a point in $X$. Consider the category $\Pi_2(X,x),$ whose objects are paths $x\rightarrow x$, and whose morphisms are homotopy classes of paths between paths, where $f,g:x\rightarrow x$. This category is a categorical group, for the proof see
	\cite{baez-lauda-2-groups} and the paper of Hardie, Kamps and Kieboom \cite{HKK}.
\end{Exam}
One can see more examples in \cite{baez-lauda-2-groups}, and also we will give them in the following paper, where we will construct a categorical group for any cssc-crossed module defined below in Section 5.

We define (strict) morphisms between categorical groups, which satisfy conditions of (strict) morphism of monoidal categories. Note that from this definition follow: $T(-x)=-{T(x)}$ and $T(-f)=-{T(f)}$, for any object $x$ and arrow $f$ in a categorical group. Categorical groups form a category with (strict) morphisms between them.
For any categorical group $\mathsf{C}=(C_0, C_1, d_0, d_1, i, m)$ denote $\Ker d_0=\{f\in \mathsf{C}|d_0(f)\approx 0\}$  and $\Ker d_1=\{f\in \mathsf{C}|d_1(f)\approx 0\}$
\begin{Lem}\label{comm} Let $\mathsf{C}=(C_0, C_1, d_0, d_1, i, m)$ be a categorical group. For any $f\in \Ker d_1$ and $g\in \Ker d_0$ we have a weak special isomorphism $f+g\approx g+f.$
\end{Lem}
\begin{proof}Suppose $d_0(g)=0'$ and $d_1(f)=0'',$ where $0'\approx 0 \approx 0''.$ By interchange law we have $(1_{0''}+g)(f+1_{0'})=f+g$ and $(g+1_{0''})(1_{0'}+f)=g+f$.  Let $\gamma \colon 0''\approx 0'$ be a special isomorphism. Applying the coherence property of a categorical group, we easily obtain that the left sides of the noted both equalities are isomorphic to $g\gamma f$, and both are weak special isomorphisms. From this it follows that there is a weak special isomorphism $f+g\approx g+f.$ \qed
\end{proof}

The analogous statement is well known for group-groupoids, where instead of the isomorphisms we have equalities in the definitions of $\Ker d_0$ and $\Ker d_1$ and in the final result \cite{BS1}.

Below we recall the definition of crossed module  introduced by Whitehead in \cite{Wth2} defining homotopy system. A {\em crossed module} $(A,B,\mu)$ consists of a group homomorphism
$\mu\colon A\rightarrow B$ together with an action $(b,a)\mapsto
{b\cdot a}$ of $B$ on $A$ such that for $a,a_1\in A$ and $b\in B$
\begin{enumerate}[leftmargin=1.5cm]
	\item [CM1.] $\mu(b\cdot a)=b+\mu(a)-b$, and
	
	\item [CM2.]  $\mu(a)\cdot a_1=a+a_1-a$.
\end{enumerate}
For an extensive treatment of crossed modules, see \cite[Part I]{Br-Hi-Si}.

Here are some examples of crossed modules.
\begin{itemize}
	\item[$\bullet$] The inclusion of a normal subgroup $N\rightarrow G$
	is a crossed module with the action by conjugation of $G$ on $N$. In particular any group $G$ can be regarded as a crossed module $1_G\colon G\rightarrow G$.
	\item[$\bullet$] For any group $G$, modules over the group ring of $G$ are crossed
	modules with $\mu= 0.$
	\item[$\bullet$] For any group $G$ the object $\mu\colon G\rightarrow \Aut G$ is a crossed module, where $\mu(g)\cdot g'=\mu(g)(g')$ for any $g, g' \in G.$
\end{itemize}

A {\em  morphism} $(f,g)\colon (A,B,\mu)\rightarrow (A^{\prime },B^{\prime },\mu^{\prime })$ of crossed module
is a pair $f\colon A\rightarrow A^{\prime }$,
$g\colon B\rightarrow B^{\prime }$ of morphisms of groups such that
$g\mu=\mu^{\prime }f$ and $f$ is an operator morphism over $g$,
i.e.,  $f(b\cdot a)=$ $g(b)\cdot f(a)$ for $a\in A$, $b\in B$. So crossed modules and  morphisms of them, with the obvious composition of morphisms $(f',g')(f,g)=(f'f,g'g)$  form a category.

\begin{Def}\label{Defgroup-groupoid}
	A {\em group-groupoid} $G$ is a \emph{group object} in the category of groupoids, which means that it is a groupoid $G$ equipped with functors
	\begin{enumerate}[label=(\roman{*}), leftmargin=1cm]
		\item  $+\colon G\times G\rightarrow G$, $(a,b)\mapsto a+b$;
		\item  $u\colon G\rightarrow G$, $a\mapsto -a$;
		\item   $0\colon \{\star\}\rightarrow G$, where $\{\star\}$ is a singleton, \qed
	\end{enumerate}
	which are called respectively sum, inverse and zero, satisfying the usual axioms for a group.
\end{Def}
The  definition we gave was introduced by Brown and Spencer in  \cite{BS1} under the name {\em $\mathcal{G}$-groupoid}, where the group operation is multiplication. The term group-groupoid was used later in \cite{BM1}. It is interesting that the group object in the category of small categories called {\em $\mathcal{G}$-category} is a group-groupoid. As it is noted by the authors this fact was known to Duskin.

\begin{Exam}
	If $X$ is a topological group, then the fundamental groupoid $\pi_1X$ of the space $X$ is a group-groupoid \cite{BS1}.
\end{Exam}
\begin{Exam}
	For a group $X$, the direct product $G=X\times X$ is a group-groupoid.  Here the domain and codomain homomorphisms are the projections; the object inclusion homomorphism is defined by the diagonal homomorphism $i(x)=(x,x),$ for any $x\in X$ and the composition of arrows is defined by $(x,y)\circ(z,t)=(z,y)$ whenever
	$x=t$, for any $x, y, z, t\in X.$
\end{Exam}

\begin{Theo}	\label{Theocatequivalence}\cite{BS1}
	The categories of crossed modules and of group-groupoids are equivalent.
\end{Theo}

According to the authors this result is due to Verdier, which was used by Duskin and which was discovered independently by them. It was proved by Porter that the analogous statement is true in more general setting of a category of groups with operations \cite{Por}.

\section{Groups up to congruence relation}

Let $X$ be a non-empty set with an equivalence relation  $R$ on $X$.  Denote such a pair by $X_R$. Define a category whose objects are the pairs $X_R$ and morphisms are functions $f\colon X_R\rightarrow Y_S,$ such that $f(x)\sim_S f(y),$ whenever $x\sim_R y$. Denote this category by $\tilde{\Sets}$.

Note that for  $X_R,Y_S\in Ob(\tilde{\Sets})$, the product  $X_R\times Y_S$ is a product object in $\Sets$ with the equivalence  relation $R\times S$
defined by  \[(x,y)\sim_{R\times S}(x_1,y_1) \Leftrightarrow x\sim_R x_1 \q \t{and} \q y\sim_Sy_1 \]

We now define \emph{group up to congruence relation} or briefly {\em c-group} concept as follows.
\begin{Def}
	Let $G_R$ be an object in $\tilde{\Sets}$ and
	\[ \begin{array}{cccl}
	m \colon &G\times G& \longrightarrow & G\\
	&(a,b)    & \longmapsto     & a+b
	\end{array}\]
	a morphism in $\tilde{\Sets}$, i.e,  $m\in \tilde{\Sets}((G\times G)_{R\times R},G_R)$. $G_R$ is called a {\em c-group} if the following axioms are satisfied.
	\begin{enumerate}[label=(\roman{*}), leftmargin=1cm]
		\item $a+(b+c)\sim_R(a+b)+c$ for all $a,b,c\in G$;
		\item  there exists an element $0\in G$ such that $a+0\sim_R a\sim_R 0+a,$ for all $a\in G;$
		\item for each $a\in G$ there exists an element $-a$ such that  $a-a\sim_R 0$ and $-a+a\sim_R0$.
	\end{enumerate}
\end{Def}

In a c-group $G$,  $0\in G$ is called  {\em zero element} and for any $a\in G$ the element $-a\in G$ is called {\em inverse} of $a$. The congruence relations involved in group like identities and their compositions and sums will be called \emph{special congruence relations}.

\begin{Rem}Let $G_R$ be a c-group.  Then we have the following:
	\begin{enumerate}[label=(\roman{*}), leftmargin=1cm]
		\item if $a\sim_R b$ and $c\sim_R d$ for $a, b, c, d\in G$ then $a+c\sim_R b+d;$
		\item if $0$ and $0'$ are  different zero elements in  $G_R$, then  $0\sim_R 0'$;
		\item if $-a$ and $a'$ are different inverses of $a\in G$, then $a'\sim_R -a$;
		\item if $a\sim_Rb$ then  $-a\sim_R -b$.
	\end{enumerate}
\end{Rem}

\begin{Exam}
	Every group $G$ is a c-group where the  equivalence relation  is equality.
\end{Exam}

So the concept of c-group generalizes group notion.

\begin{Exam}
	This example comes from the Mac Lane's paper \cite{Mac Lane}, where the author  regards the quotient group as a group with congruence relation. Let $G$ be a group and $H$ a normal subgroup in $G.$ The quotient group $G/H$ can be regarded as a group with the same elements as the group $G$ and with the congruence relation - $g\sim g'$ if and only if $g-g'\in H.$ The operations are defined in the same way as in $G$ and they preserve the congruence relation. Such a group is a c-group, where group identities are satisfied up to equality.
\end{Exam}

\begin{Exam}
	Let $X$ be a topological space and $x\in X$. The set $P(X,x)$  of all  closed paths at $x$ is a c-group with the composition of paths. Here the congruence relation is the homotopy of the paths.
\end{Exam}

\begin{Exam}
	Let $\mathbb{Z}^{\ast}=\mathbb{Z}\backslash\{0\}$. Define an equivalence relation on  $\mathbb{Z}^{\ast}$ by $x\sim_Ry  \Leftrightarrow  xy>0$.  Then $\mathbb{Z}^{\ast}$ becomes a c-group with respect to the multiplication. The unit is the number $1$ and the inverse for any number is itself this number.
\end{Exam}

\begin{Exam}
	In a categorical group $\mathsf{C}$ the set $C_1$ of morphisms and the set $C_0$ of objects are both c-groups. The congruence relations are isomorphisms between arrows and between objects respectively.
\end{Exam}

\begin{Exam}\label{gr}  Any group can be endowed with a c-group structure. To show this recall that every group $G$ can be regarded as a part of a certain crossed module in the category of groups, for example $G \rightarrow \Aut G.$ According to Theorem \ref{Theocatequivalence} there exists  a group-groupoid $C=(C_0, C_1, d_o, d_1, i, m),$ for which $d_1|_{\Ker d_0}\colon \Ker d_0 \rightarrow C_0,$ is a crossed module and is isomorphic to $G \rightarrow \Aut G.$ $\Ker d_0$ is a c-group, the congruence relation on it is induced from the congruence relation in $C_1,$ which is the relation being isomorphic between the morphisms. From this follows that $G$ has also a c-group structure, group identities are satisfied up to equality, and naturally all special isomorphisms are equalities.
\end{Exam}

\begin{Def}
	Let $G_R$ be a c-group. If  $a+b\sim_R b+a$ for all $a,b\in G$, then  $G_R$ is called {\em c-abelian (or c-commutative)} c-group.
\end{Def}

\begin{Def}
	Let $G_R$ and  $H_S$  be  c-groups. A morphism   $f\in \tilde{\Sets}(G_R,H_S)$ such that  $f(a+b)=f(a)+f(b)$  for any $a,b\in G$ is called a  {\em c-group morphism} from  $G_R$ to $H_S$.
\end{Def}

From the definition it follows that a morphism between c-groups preserves the congruence relation; moreover we obtain that $f(0)\sim 0$ and $f(-a)=-f(a)$, for any $a\in G,$ where the second equality means that $f(-a)$ is one of the inverse element of $f(a).$ As a result we obtain that a morphism between c-groups carries special congruence relations to special congruence relations

\begin{Rem}
	If $f\colon G_R\rightarrow H_S$ and  $g\colon H_S\rightarrow N_T$ are  two c-group morphisms, then $gf\colon G_R\rightarrow N_T$ is also a c-group morphism. Further for each c-group  $G_R$ there is a unit morphism $1_G\colon G_R\rightarrow G_R$ such that $1_G$ is the identity function on $G_R$. Therefore we have a category of c-groups with c-group morphisms; denote this category by $\cGr$.
\end{Rem}

Let $G_R$ and  $H_S$ be  c-groups, and  $f\colon G_R\rightarrow H_S$ a morphism of c-groups.

\begin{Def}
	The subset $\cKer f=\{a\in G_R~|~ f(a)\sim_S 0_H \}$ is said to be {\em c-kernel} of the c-group morphism $f$.
\end{Def}

Note that $\cKer d_0$ is a c-group with the congruence relation induced from the isomorphisms in $C_1$.

\begin{Def}
	The subset  $\cIm f=\{b\in H_S~|~  \exists a\in G_R,  f(a)\sim_S b\}$ is said to be the {\em c-image} of the morphism $f$.
\end{Def}

\begin{Lem}\label{grp}
	Let $G $ be a  c-group with congruence relation $R$.  Then the quotient set $G/R$ becomes a group with the operation  defined by the induced map
	\[ \begin{array}{cccl}
	m^{\ast} \colon & G/R\times G/R & \longrightarrow & G/R           \\
	&   ([a],[b])   &   \longmapsto   & [a]+[b]=[a+b]
	\end{array}\]
	
\end{Lem}

\begin{Def}
	Let $G_R$ be a c-group and $H$ be a subset of the underlying set of $ G$.  $H$ is called a {c-subgroup} in $G_R$ if $H_S$ is a c-group with the addition and congruence relation $S$ induced from $G_R$.
\end{Def}

Let $G_R$ be a c-group and $H$ be a subset of $G$. If for an element $a\in G$ there exists an element $b\in H$ such that $a\sim_R b$ then we write $a\inc H$.
If $H$ and $H'$ are two subsets of $G_R$, then we write $H\tilde{\subset}H'$ if for any $h\in H$ we have $h \inc H'$. If $H\tilde{\subset}H'$ and $H'\tilde{\subset}H,$ then we write $H\sim H'$.

\begin{Def}
	Let $G_R$ be a c-group and $H_S \subseteq G_R $  a c-subgroup in $G_R$. Then $H_S$ is called \emph{normal c-subgroup} if $g+h-g\inc H_S$ for any $h\in H_S$ and  $g\in G.$
\end{Def}

The condition given in the definition is equivalent to the condition $g+H_S-g\tilde{\subset} H_S,$ and it is equivalent itself to the condition $g+H_S\sim H_S+g$ for any $g\in G.$

\begin{Def}
	Let $G_R$ be a c-group and $H_S \subseteq G_R $ be a c-subgroup in $G_R$. Then $H_S$ is called \emph {perfect c-subgroup} if from $g\inc H$ it follows that $g\in H,$ for any $g\in G.$
\end{Def}

\begin{Def}
	A c-group $G_R$ is called \emph {connected} if $g\sim g'$ for any $ g, g'\in G.$
\end{Def}

\begin{Lem}\label{connected}
	Let $G_R$ and  $H_S$ be  c-groups and let  $f\colon G_R\rightarrow H_S$ be a morphism of c-groups. Then
	\begin{enumerate}[label=(\roman{*}), leftmargin=1cm]
		\item $\cKer f$ is perfect and normal c-subgroup in $G_R$;
		\item $\cIm f$ is perfect c-subgroup in $H_S$.
	\end{enumerate}
\end{Lem}

\begin{proof}
	Follow from the definitions.\qed
\end{proof}

Now we shall construct the quotient object $G/H$, where $H$ is a normal c-subgroup of a c-group $G$. Consider the classes  $\{g+H|g\in G\}.$ If $g+H \cap g'+H\neq \emptyset,$ then we obtain $-g+g'\inc H,$ which implies that $g+H\sim g'+H.$ Now consider classes of these classes $\{cl(g+H)|g\in G\},$ where $cl(g+H)=\cup\{x\in G|x\inc g+H\}.$ We define $G/H=\{cl(g+H)|g\in G\}.$ An addition operation in this set is defined by $cl(g+H)+cl(g'+H)=cl((g+g')+H),$ for any $g,g'\in G.$ It is easy to see that this operation is defined correctly, it is associative and we have the unit element $cl(0+H)$. Actually constructed object is a group, the congruence  relation in $G/H$ is the equality $``=".$ We have a usual surjective morphism $p: G \rightarrow G/H.$

\begin{Lem} (i) If $G$ is a c-group and $H$ is a normal c-subgroup in $G,$ then for any group $G'$ and c-group morphism $f:G\rightarrow G'$, if $f(h)=0$ for any $h\in H,$ there exists a unique morphism $\theta :G/H \rightarrow G',$ in $\cGr$  such that $\theta p=f$.
	(ii)  If $H$ is a perfect normal c-subgroup in $G,$ then $H=\cKer p.$
\end{Lem}
\begin{proof}
	Easy checking.\qed
\end{proof}

\section{Actions and crossed modules in $\cGr$}

An extension in the category $\cGr$ is defined in a similar way as in the category of groups.

\begin{Def}
	Let $A$, $B\in\cGr$. An \emph{extension} of $B$ by $A$ is a sequence
	\begin{equation} \label{extension}
	\xymatrix{0\ar[r]&A\ar[r]^-{i}&E\ar[r]^-{p}&B\ar[r]&0}
	\end{equation} in which $p$ is surjective  and  $i$ is the c-kernel of $p$ in $\cGr$. We say that an extension is \emph{split}  if there is a morphism $s\colon B\rightarrow E,$ such that $ps=1_B$.
\end{Def}

We shall identify $a\in A$ with its image $i(a)$. We  shall use the notation $b\cdot a=s(b)+(a-s(b))$. Then a split extension induces an action (from the left side) of $B$ on $A$. We have the following conditions for this action:
\begin{enumerate}[label=(\roman{*}), leftmargin=1cm]
	\item $b\cdot(a+a_1)\sim (b\cdot a)+(b\cdot a_1)$,
	\item $(b+b_1)\cdot a\sim b\cdot(b_1\cdot a)$,
	\item $0\cdot a\sim a$,
	\item If $a\sim a_1$ and $b\sim b_1$ then $b\cdot a \sim b_1\cdot a_1$,\end{enumerate}
for $a,a_1\in A$ and $b,b_1\in B$.

Here and in what follows we omit congruence relations symbols for $A$ and $B$.

Let $A, B \in \cGr$ and suppose $B$ acts on $A$. Consider the product $B \times A$ in  $\cGr$. We have the operation in $B \times A$, defined in analogous way as in the case of groups:

$(b',a')+(b,a)=(b'+b, a'+b'\cdot a)$
for any $b, b'\in B, a, a'\in A$.

This operation is associative up to the relation defined by $(b,a)\sim (b',a')$ if and only if $b\sim b'$ and $a\sim a'$, which is a congruence relation.
Obviously we have a zero element $(0,0)$ in $B \times A$ and the opposite element for any pair $(b,a)\in B \times A$ is $(-b, -b\cdot (- a)).$ Therefore we have a semidirect product $B \ltimes A$ in $\cGr$.

\begin{Def}\label{c-iso}
	Let $f:D\rightarrow D'$ be a morphism in $\cGr$. $f$ is called an isomorphism up to congruence relation or c-isomorphism if there is a morphism $f':D'\rightarrow D$, such that $ff'\sim 1_{D'}$ and $f'f\sim 1_D$.
\end{Def}

We will denote such isomorphism by $\tilde{\approx}.$

We have a natural projection $p': B \ltimes A \rightarrow B$. The c-kernel of $p'$ is not isomorphic to $A$ as it is in the case of groups, but we have an isomorphism up to congruence relation $\cKer p'\tilde{\approx}A$.

Let $0\rightarrow A\rightarrow E \rightarrow B\rightarrow 0$ be a split extension of $B$ by $A$ in $\cGr$. Then we have an action of $B$ on $A$ and the corresponding semidirect product $B \ltimes A$. In this case we obtain a c-isomorphism $E\tilde{\approx}B\ltimes A$ given by the correspondences analogous to the group case.

\begin{Def} Let $G$ and $H$ be two c-groups, $\partial\colon G\rightarrow H$ morphism of c-groups and $H$ acts on $G$. We call $(G,H,\partial)$ a \emph{c-crossed module} if the following conditions are satisfied:
	\begin{enumerate}[label=(\roman{*}), leftmargin=1cm]
		\item $\partial(b\cdot a)= b+(\partial(a)-b)$,
		\item $\partial(a)\cdot a_1\sim a+(a_1-a).$
	\end{enumerate}
	for $a,a_1\in G$ and $b\in H$.
\end{Def}

Let $(G,H,\partial)$ and $(G',H',\partial')$ be two c-crossed modules. A \emph{c-crossed module morphism} is a pair of morphisms $<f,g>\colon(G,H,\partial)\rightarrow(G',H',\partial')$ such that the diagram
\[\xymatrix{
	G \ar[r]^-{\partial} \ar[d]_-{f} &  H \ar[d]^-{g} \\
	G'\ar[r]^-{\partial'}  &  H'   }\]
is commutative, and for all $b\in H$ and $a\in G$, we have $f(b\cdot a)= g(b)\cdot f(a)$, where $f$ and $g$ are morphisms of c-groups.

c-crossed modules and morphisms of c-crossed modules form a category.

\begin{Exam} Any crossed module in the category of groups can be endowed with the structure of c-crossed module, the proof is analogous to the one in Example \ref{gr}.
\end{Exam}
For other examples see Section 5.

Let $G\in \cGr$ and $H$ be a normal c-subgroup in $G$. It is easy to see that in general we do not have a usual action by conjugation of $G$ on $H.$

\begin{Lem}
	If $H$ is a perfect normal c-subgroup of a c-group $G$, then we have an action of $G$ on $H$ in the category $\cGr$ and the inclusion morphism defines a c-crossed module.
\end{Lem}

\begin{proof}
	Easy checking.\qed
\end{proof}

For a categorical group $\mathsf{C}=(C_0, C_1, d_0, d_1, i, m)$, we have a split extension
\begin{equation} \label{extension}
\xymatrix{0\ar[r]&\cKer d_0\ar[r]^-{j}&C_1\ar[r]^-{d_0}&C_0\ar[r]&0}
\end{equation}
where $\cKer d_0$ is a c-group, a congruence relation is defined naturally as an isomorphism between the arrows in $\cKer d_0$. Now we define an action of $C_0$ on $\cKer d_0$ by
\[ \begin{array}{ccl}
C_0\times \cKer d_0& \rightarrow & \cKer d_0\\
(r,c)              & \mapsto     & r\cdot c=i(r)+(j(c)-i(r)).
\end{array}\]
\begin {Prop}
The action of $C_0$ on $\cKer d_0$ satisfies action conditions in $\cGr$.
\end {Prop}
\begin{proof}First we shall show the congruence relation $r\cdot(c+c')\sim r\cdot c+r\cdot c'.$
	\begin{alignat*}{2}
	r\cdot(c+c') &= i(r)+((c+c')-i(r)) \\
	& \sim (i(r)+((c-i(r))+i(r))+(c'-i(r)))  \\
	& \sim (i(r)+(c-i(r)))+(i(r)+(c'-i(r))) \\
	& = r\cdot c+r\cdot c'.
	\end{alignat*}
	
	Next we shall show that $(r+r')\cdot c\sim r\cdot (r'\cdot c)$. We shall omit some brackets since we deal with congruence relation.
	
	We have
	\begin{alignat*}{2}
	(r+r')\cdot c &= i(r+r')+(c-i(r+r')) \\
	& \sim i(r)+(i(r')+c-i(r'))-i(r) \\
	& =r\cdot (r'\cdot c).
	\end{alignat*}
	
	It is trivial that $0\cdot c\sim c$ and  $r\cdot 0\sim 0$. Now we shall show that if $r\sim r'$ and $c\sim c'$ then $r\cdot c\sim r'\cdot c'$. We have
	\begin{alignat*}{2}
	r\cdot c &= i(r)+(c-i(r))\\ 
	&\sim i(r')+(c'-i(r'))\\ 
	&=r'\cdot c'.
	\end{alignat*} \qed
\end{proof}

\section{cssc-crossed modules and the main theorem}

\begin{Def}\label{connected cr}
	A c-crossed module $(G,H,\partial)$ will be called connected if $G$ is a connected c-group.
\end{Def}

Denote $d=d_1{|\cKer d_0}.$

\begin{Prop}
	For a categorical group $\mathsf{C}=(C_0, C_1, d_0, d_1, i, m)$, $(\cKer d_0,C_0,d)$ is a connected c-crossed module.
\end{Prop}

\begin{proof} $\cKer d_0$ is a connected c-group, which follows from the fact  that any two arrows in $\cKer d_0$ have the domains isomorphic to $0$ and that $\mathsf{C}$ is a groupoid.  Note that the congruence relation in $C_0$  is generated by the isomorphisms between the objects in $C_0.$ Therefore $d$ preserves the congruence relation in  $\cKer d_0$ since $f\approx f'$ implies that $d_1 f\approx d_1 f'.$ For the first condition of crossed module we have $d(r\cdot c)=d(i(r)+(c-i(r)))=r+(dc-r),$ for any $c\in\cKer d_0$ and $r\in C_0.$
	
	For the second condition of crossed module we have to prove that $(dc)\cdot c'\sim c+(c'-c)$ for $c, c'\in \cKer d_0,$ which follows from the fact that $\cKer d_0$ is a connected c-group. \qed
\end{proof}

Now we shall introduce another sort of object denoted as $c\textbf{K}er d_0$ for any categorical group $\mathsf{C}=(C_0, C_1, d_0, d_1, i, m)$. By definition $c\textbf{K}er d_0 =\{f\in C_1| d_0(f)=0\}$. An addition operation is defined by $f+f'=(f+f')\gamma,$ where $f+f':0+0\rightarrow d_1(f)+d_1(f')$ is a sum in $C_1$, i.e. the same as the sum in $\cKer d_0$ and $\gamma : 0 \rightarrow 0+0$ is a unique special isomorphism in $C_1$).  $\sim$ -relation in $c\textbf{K}er d_0$ is induced from the relation in $C_1$, which is a relation of being isomorphic in $C_1$ and it is a congruence relation in $c\textbf{K}er d_0$. It is obvious that $d\mid_{c\textbf{K}er d_0}$ preserves the congruence relation. The operation in $c\textbf{K}er d_0$ is associative up to congruence relation. Zero element in $c\textbf{K}er d_0$ is a zero arrow $0$; we have $f+0\sim f,$ $0+f\sim f,$ for any $f\in c\textbf{K}er d_0.$ The opposite morphism of $f$ in $C_1$ is $-f: -0\rightarrow -d_1f.$ There is a unique special isomorphism $\kappa: 0\approx -0$. Define the opposite morphism $-f$ in $c\textbf{K}er d_0$ as $-f\kappa.$  One can easily see that $f+(-f)\approx 0$ in $c\textbf{K}er d_0$ and   the $\sim$ -relation is a congruence relation. Therefore $c\textbf{K}er d_0$ is a c-group. $C_0$ is also a c-group, where a congruence relation is generated by isomorphisms between the objects. Now we will define an action of $C_0$ on $c\textbf{K}er d_0.$ By definition $r\cdot c=(i(r)+(c-i(r)))\gamma$ for any $r\in C_0, c\in c\textbf{K}er d_0,$ where $\gamma$ is a special isomorphism $0\approx r+(0-r),$ which is unique as we know already. Here we check  action identities. We have 
\begin{alignat*}{2}
r\cdot (c+c')&=(i(r)+((c+c')-i(r))\gamma \\
& \sim(i(r)+(c-i(r)))\gamma_1+ i(r)+(c'-i(r)))\gamma_2 \\
& =r\cdot c+ r\cdot c'
\end{alignat*}
for any $r\in C_0, c,c' \in c\textbf{K}er d_0$. Other three conditions of action for c-groups are checked in analogous way.

\begin{Def}\label{strict cr}
	A c-crossed module $(G,H,\partial)$ will be called strict if it satisfies c-crossed module conditions, where instead of $\sim$-relation in the second condition is equality, i.e.
	\begin{enumerate}[label=(\roman{*}), leftmargin=1cm]
		\item $\partial(b\cdot a)= b+(\partial(a)-b)$,
		\item $\partial(a)\cdot a_1= a+(a_1-a)$,
	\end{enumerate}
	for $a,a_1\in G$ and $b\in H$.
\end{Def}

\begin{Def}\label{weak special} In a c-crossed module $(G,H,\partial)$ a congruence relation $g\sim g'$ in $G$ will be called weak special relation if $\partial (g)\sim\partial(g')$ is a special congruence relation in $H$.
\end{Def}

Since in a c-crossed module $(G,H,\partial)$ the morphism $\partial$ carries any special congruence relation to the special congruence relation, in a crossed module every special congruence relation in $G$ is a weak special congruence relation.
Let $(G,H,\partial)$ be a c-crossed module.

\begin{Cond}\label{special cr}
	For any congruence relation $\gamma :\partial c\sim r$,  there exists $c'\sim c$, such that $\partial c' =r$, where $c, c'\in G$ and $r\in H$. If $\gamma $ is a special congruence relation, then  $c'$ is a unique element in $G$  which is weak equivalent to $c$.
\end{Cond}

\begin{Def}\label{special}
	A c-crossed module will be called \emph{special} if it satisfies Condition \ref {special cr}.
\end{Def}

If a c-crossed module is connected, strict and special we will write shortly that it is a cssc-crossed module. This kind of crossed modules are exactly those we were looking for the description of categorical groups up to equivalence of the corresponding categories, which will be proved in the next paper.

\begin{Theo}
	For a categorical group $\mathsf{C}=(C_0, C_1, d_0, d_1, i, m)$ the triple $(c\textbf{K}er d_0, C_0, d)$ is a cssc-crossed module.
\end{Theo}

\begin{proof}
	First we shall show that we have equality in the first condition of the crossed module $(c\textbf{K}er d_0, C_0, d)$. We have $d(r\cdot c)=d((i(r)+(c-i(r))\varepsilon )=d(i(r)+(c-i(r)))=r+(dc-r),$ where $\varepsilon: 0\rightarrow r+(0-r)$ is a special isomorphism. Now we shall show that we have equality in the second condition of a crossed module. First we compute the left side of the condition. We have $dc\cdot c'=(i(dc)+(c'-i(dc)))\gamma$ where $\gamma: 0\rightarrow dc+(0-dc)$ is a special isomorphism and $i(dc)+(c'-i(dc))$ is a morphism $dc+(0-dc)\rightarrow dc+(dc'-dc)$. We have $-c+idc\in \cKer d_1$ and $c'\in \cKer d_0;$ by Lemma \ref{comm} we obtain that there is a weak special isomorphism $(-c+idc)+c'\approx c'+(-c+idc);$ from this it follows that $idc+c'\approx c+c'-c+i(dc),$ which implies $id(c)+(c'-i(dc))\approx c+(c'-c);$ from which we obtain the  weak special isomorphism $id(c)\cdot c'\approx c+(c'-c).$  By the definition of a sum in $c\textbf{K}er d_0$ for the right side we have $c+(c'-c)=(c+(c'-c)\varphi)\psi,$  where $\varphi:0\rightarrow 0-0$ and $\psi:0\rightarrow 0+0$ are special isomorphisms. Here we have in mind that $d(-c)=-dc$ and $i(-dc)=-i(dc).$ Obviously we have a weak special isomorphism $(c+(c'-c)\varphi\approx i(dc)+(c'-i(dc))$, from which it follows that there is a special isomorphism between the domains of these morphisms $\theta: 0+0\rightarrow dc+(0-dc),$ such that $(i(dc) +(c'-i(dc)))\theta=c+(c'-c)\varphi.$ Here we applied that the codomains of these morphisms are equal. Since $\psi, \theta $ and $\gamma$ are special isomorphisms, we obtain that $\theta \psi=\gamma,$ from this it follows that $(c+(c'-c)\varphi)\psi=(i(dc)+(c'-i(dc)))\gamma,$  which means that for the c-crossed module $(c\textbf{K}er d_0, C_0, d)$ we have an equality in the second condition for c-crossed modules.
	
	The crossed module is connected by the definition of the object $c\textbf{K}er d_0$. Now we shall prove that this crossed module is a special c-crossed module. Let $c\in c\textbf{K}er d_0,$ and there is a congruence $\gamma: dc\sim r$, which means that $\gamma$ is an isomorphism in $C_1$. Take $c'=\gamma c$, then we will have $c'\approx c$ in $C_1,$ which means that $c'\sim c$ in $c\textbf{K}er d_0.$ Suppose $\gamma$ is a special congruence relation, then it is a special isomorphism in $C_1.$ From the coherence property of $\mathsf{C}$ we have that $\gamma$ is a unique special isomorphism from $dc$ to $r$ and therefore there is a unique morphism $d_0c\rightarrow r$, which is weak special isomorphic to $c$ and it is a composition $\gamma c$. Therefore $c'$ is unique with this property and $(c\textbf{K}er d_0, C_0, d)$ is a special c-crossed module. \qed
\end{proof}

\section*{Acknowledgements}
	The first author is grateful to Ercyies University (Kayseri, Turkey)  and Prof. Mucuk for invitations and to the Rustaveli National Science Foundation for financial support, grant GNSF/ST09 730 3 -105.

%
%



\end{document}